\documentclass[11pt,reqno]{amsart}
\usepackage{amsfonts,amssymb,amsmath}
\usepackage{ytableau}
\usepackage{nicematrix}
\usepackage{tikz}
\usepackage{xcolor}
\usepackage{mathtools}
\usepackage{multirow}
\usepackage{float}

\definecolor{fgreen}{RGB}{44,144, 14}

\renewenvironment{proof}{{\bfseries Proof.}}{\qed}

\numberwithin{equation}{section} 
\newtheorem{theorem}{Theorem}[section] 
\newtheorem{proposition}[theorem]{Proposition} 
 
\newtheorem{lemma}[theorem]{Lemma} 

\theoremstyle{definition}
\newtheorem{definition}[theorem]{Definition} 
\newtheorem{remark}[theorem]{Remark} 
\newtheorem{example}[theorem]{Example}
\usepackage[
colorlinks=true,citecolor=cyan,  urlcolor=blue,linkcolor=blue]{hyperref}

\def\O{\mathbb O}

\def\R{\mathbb R}

\def\C{\mathbb C}

\def\F{\mathbb F}

\def\H{\mathbb H}

\def\ib{\mathbf {i}}
\def\jb{\mathbf {j}}
\def\kb{\mathbf {k}}

\def\d{\mathbf{ d}}

\def\N{\mathbb N}

\def\E{\mathbb E}

\def\R{\mathbb R}

\def\d{\mathcal D}

\newcommand{\SL}{\mathrm{SL}}

\newcommand{\GL}{\mathrm{GL}}

\def\P{\mathbb P}

\def\d{\mathcal D}

\def\R{\mathbb {R}}
\def\C{\mathbb {C}}
\def\N{\mathbb {N}}
\def\H{\mathbb {H}}

\def\O{\mathbb {O}}
\def\E{\mathbb {E}}
\def\F{\mathbb {F}}

\def\d{\mathbf {d}}

\def\ib{\mathbf {i}}

\def\jb{\mathbf {j}}

\def\GL{\rm GL}
\def\SL{\rm SL}

\newcommand{\defref}[1]{Definition~\ref{#1}}

\newcommand{\thmref}[1]{Theorem~\ref{#1}}
\newcommand{\lemref}[1]{Lemma~\ref{#1}}
\newcommand{\remref}[1]{Remark~\ref{#1}}
\newcommand{\propref}[1]{Proposition~\ref{#1}}

\begin{document}

\title[Reversibility in $\mathrm{SL}(n,\H) \,$ and $\,\mathrm{PSL}(n,\H)$]{Product of two involutions in 
	 quaternionic special linear group}
 \author[K. Gongopadhyay,    T. Lohan and  C. Maity]{Krishnendu Gongopadhyay, 
 Tejbir Lohan and Chandan Maity}

\address{Indian Institute of Science Education and Research (IISER) Mohali,
 Knowledge City,  Sector 81, S.A.S. Nagar 140306, Punjab, India}
\email{krishnendu@iisermohali.ac.in, krishnendug@gmail.com}

\address{Indian Institute of Technology Kanpur, Kanpur 208016, Uttar Pradesh, India}
\email{tejbirlohan70@gmail.com, tejbir@iitk.ac.in}

\address{Indian Institute of Science Education and Research (IISER) Berhampur, 
	Berhampur  760003, Odisha, India}
\email{cmaity@iiserbpr.ac.in}

\makeatletter
\@namedef{subjclassname@2020}{\textup{2020} Mathematics Subject Classification}
\makeatother




\begin{abstract} 
An element of a group is called reversible if it is conjugate to its own inverse. Reversible elements are closely related to strongly reversible elements, which can be expressed as a product of two involutions. In this paper, 	
we classify the reversible and strongly reversible elements in the quaternionic special linear group $ \mathrm{SL}(n,\mathbb{H})$ and quaternionic projective linear group $ \mathrm{PSL}(n,\mathbb{H})$.  We prove that an element of $ \mathrm{SL}(n,\mathbb{H})$ (resp. $ \mathrm{PSL}(n,\mathbb{H})$) is reversible if and only if it is a product of two skew-involutions (resp. involutions). 
\end{abstract}

\maketitle 
\setcounter{tocdepth}{1}

\section{Introduction} 
Symmetry is a widely studied concept in various areas of science where dynamical systems or geometric automorphisms are involved. Symmetries are formulated mathematically using the notion of groups.  
Let $G$ be a group. An element $\iota$ of order at most two in $G$ is called an involution, i.e., $\iota^2=e$. It is a problem of potential applications in several areas of mathematics to decompose an element into a product of involutions; see \cite{Wo, Dj, HP,  El, BM,  RV}. Particularly interesting elements are those that are products of two involutions. In the literature,  an element $g \in G$ is called \textit{reversible} or \textit{real} if $g$ is conjugate to $g^{-1}$ in $G$.   An element $g \in G$ is called \textit{strongly reversible} or \textit{strongly real} if $g$ is conjugate to $g^{-1}$ by an involution in $G$. Equivalently, an element $g \in G$ is strongly reversible if and only if it can be written as a product of two involutions in $G$; see \cite{OS}. An element that conjugates $g$ to $g^{-1}$ is also known as `reversing symmetry' of $g$; see \cite{LR, BR}.  We refer to \cite{Ba} for a brief survey on symmetries and reversing symmetries in dynamical systems.

By reversibility in a group $G$, we mean a classification of reversible and strongly reversible elements in $G$. 
The reversible and strongly reversible elements appear in various areas of mathematics, such as classical dynamics,  mechanics, group theory, representation theory, geometry, complex analysis, and functional equations; see \cite{AA, LR, TZ,  O'Fa}. Classifying the reversible and strongly reversible elements in a group has been a problem of broad interest; see the monograph \cite{OS} for a survey on this theme. In spite of many works, complete and concrete classifications of reversible and strongly reversible elements are not known for many families of groups. 
The aim of this article is to investigate reversibility in the group ${\rm SL}(n,  \H)$,  $n \times n$ matrices over the skew-field $\H$ of Hamilton's quaternions with quaternionic determinant one. 

The group of projective automorphisms of the $n$-dimensional quaternionic projective space $ \P_{\H}^n$ is  the projective linear group ${\rm PSL}(n, \H)= \mathrm{SL}(n,\H)/\{\pm \mathrm{I}_{n}\}$.  When $n=2$, this is related to the geometry of the five-dimensional hyperbolic space or the M\"obius geometry of the four-dimensional sphere; see \cite{Go1, PS}. Reversibility is a natural problem in this context from a geometric point of view. Using the action of ${\rm SL}(2, \H)$ on  $ \P_{\H}^1$ by M\"obius transformations,   Lávička, O'Farrell, and Short classified the reversible and strongly reversible elements in ${\rm PSL}(2, \H)$ in \cite{LOS}. This work was generalized to the M\"obius group ${\rm SO}_o(n,1)$ independently in the works \cite{Go2} and \cite{Sh1}, also see \cite{BM}. However, not much has been investigated for the problem in ${\rm SL}(n, \H)$. On the other hand, reversibility in ${\rm SL}(n, {\rm \F})$, where ${\rm \F}$ is a field, has been investigated by many authors,  e.g., \cite{Wo, ST,  GS, GM}, and also see  \cite[p. 77]{OS}. In spite of these works, a complete list of the reversible and strongly reversible elements in ${\rm SL}(n, \C)$ has been obtained only recently by the authors in \cite{GLM2}.

It is a natural question to ask for reversibility in ${\rm SL}(n, \H)$ and ${\rm PSL}(n, \H)$ generalizing these works. In \cite{DGL}, reversible and strongly reversible elements in ${\rm PSL}(3, \H)$ have been classified. They were further applied to provide an algebraic characterization of the dynamical types of ${\rm PSL}(3, \H)$, which generalizes the results in the context of ${\rm PSL}(2, \H)$, e.g., in \cite{Go1, PS}.

In this paper, we solve the reversibility problem for the group ${\rm SL}(n, \H)$ for arbitrary $n$. The following theorem extends  \cite[Theorem 1.1]{DGL} and also the work \cite{LOS}. 

\begin{theorem}\label{thm-main-equiv-PSL(n,H)}
	An element of $ \mathrm{PSL}(n,\mathbb{H})$ is reversible if and only if it is strongly reversible.
\end{theorem}

Note that every element $[g]$ in ${\rm PSL}(n, \H)$ has two lifts, $g$ and $-g$, in ${\rm SL}(n, \H)$. Therefore, an element $[g]  \in {\rm PSL}(n, \H)$ is conjugate to its inverse if and only if either $g$ is conjugate to $g^{-1}$ or $g$ is conjugate to $-g^{-1} $ in ${\rm SL}(n, \H)$. To classify reversible (resp. strongly reversible) elements in ${\rm PSL}(n, \H)$, we first need to classify reversible (resp. strongly reversible) elements in ${\rm SL}(n, \H)$. However, this alone does not completely classify these elements in ${\rm PSL}(n, \H)$. Following \cite{DGL}, to obtain a complete classification of reversible (resp. strongly reversible) elements in ${\rm PSL}(n, \H)$,  investigation of the following equation is required:  $hgh^{-1}=-g^{-1}$, where $g, h \in {\rm SL}(n, \H)$. 

We investigated the reversibility in the group $\mathrm{GL}(n,\H)$ in \cite{GLM1}. Note that if two matrices are conjugate by an element of  $\mathrm{GL}(n,\H)$, then by a suitable scaling of the conjugating element,  we can assume that both the matrices are conjugate by an element of  $ \mathrm{SL}(n,\H)$, cf. \cite[Remark 2.4]{DGL}. Therefore,  the classification of reversible elements in $ {\rm SL }(n,\H)$ reduces to that of  $ {\rm GL }(n,\H)$ and is given by the following result. 
We refer to \lemref{lem-Jordan-M(n,H)} for the Jordan decomposition of quaternionic matrices.

\begin{theorem} [{\cite[Theorem 5.1]{GLM1}}] \label{thm-rev-SL(n,H)} 
	An element $A \in  {\rm SL }(n,\H)$   is reversible if and only if the Jordan blocks in the Jordan decomposition of $A$ can be partitioned into pairs $ \{ \mathrm{J}(\lambda, s),\mathrm{J}(\lambda^{-1}, s)\} $ or singletons $\{\mathrm{J}(\mu, t  )\}$,  where $\lambda, \mu \in \C \setminus \{0\}$  with non-negative imaginary parts such that  $|\lambda| \neq 1$ and $  |\mu| = 1$.  
\end{theorem}

Note that if $g \in {\rm GL }(n,\H)$ is an involution, then $ g \in {\rm SL }(n,\H)$; see \cite[Theorem 5.9.2]{rodman}. Therefore, if an element of ${\rm SL }(n,\H)$ is strongly reversible in ${\rm GL }(n,\H)$, then it will also be strongly reversible in ${\rm SL }(n,\H)$. This implies that the classification of strongly reversible elements in ${\rm SL }(n,\H)$ follows from the corresponding classification in  ${\rm GL }(n,\H)$. Unlike the general linear group $\mathrm{GL}(n,\mathbb{F})$ over a field $\F$, not every reversible element of $\mathrm{GL}(n,\mathbb{H})$ is strongly reversible in $\mathrm{GL}(n,\mathbb{H})$; see 
\cite{El}. For example, while $\ib \in \mathrm{GL}(1,\mathbb{H})$ is reversible (because $\jb \ib \jb^{-1} =-\ib$), it is not strongly reversible since  $\{\pm 1 \}$ are only two involutions in  $\mathrm{GL}(1,\mathbb{H})$.   In \cite[Theorem 5.4]{GLM1}, we gave a sufficient criterion for strong reversibility of the reversible elements in $\mathrm{GL}(n,\mathbb{H})$. In this article, we will show that these conditions are also necessary. 
We prove the following result, which classifies the strongly reversible elements in ${\rm SL}(n,\H)$, and hence in ${\rm GL}(n,\H)$. 

\begin{theorem}\label{thm-strong-rev-SL(n,H)}
	Let $A \in {\rm SL}(n,\H)$ be a reversible element. Then $A$ is strongly reversible in ${\rm SL }(n,\H)$ if and only if, in the Jordan decomposition of $A$,  every Jordan block corresponding to the non-real eigenvalue classes of unit modulus has even multiplicity.  
\end{theorem}

It is worth mentioning that \thmref{thm-rev-SL(n,H)} and \thmref{thm-strong-rev-SL(n,H)}, combined with \cite[Theorem 1.1]{GLM3}, provide a complete list of the reversible and strongly reversible elements in the affine group $\mathrm{GL}(n, \H) \ltimes \H^n$.

Recall that an element $g \in {\rm SL}(n,\H)$ is called a  skew-involution if  $ g^2 =-\mathrm{I}_n$. We refer to \cite{PaSa, JP} for the decomposition of matrices over a field as products of skew-involutions. Note that every reversible element of ${\rm SL}(n,\H)$ is not strongly reversible. However, we have the following result.

\begin{theorem}\label{thm-rev-SL-prod-inv-PSL}
	An element of   $ {\rm SL }(n,\H)$ is reversible if and only if it can be written as a product of two skew-involutions in ${\rm SL }(n,\H)$.
\end{theorem}

In this paper, our approach uses the notion of  \textit{Weyr canonical form} of a matrix (see \cite{Sh, COV}), 
and  the  structure of  the  \textit{reversing symmetry group} or \textit{extended centralizer}  $\mathcal{E}_{\mathrm{GL}(n,\mathbb{H})}(A) := \{ g \in \mathrm{GL}(n,\mathbb{H}) \mid gAg^{-1} = A \hbox{ or } gAg^{-1} = A^{-1} \}$ for  $A \in \mathrm{GL}(n,\mathbb{H})$;   see  \cite{BR}, also  \cite[Section 2.1.4]{OS}.  
We classified reversibility in the group  ${\rm SL}(n,\C)$ in \cite{GLM2} and also described reversing symmetries for certain types of Jordan forms in ${\rm GL}(n,\C)$; see \cite[Table 1]{GLM2}. 
In this article,  an analog of \cite[Table 1]{GLM2} in the context of ${\rm GL}(n,\H)$ has been described in  Table \ref{table:1}; see Section \ref{sec-proof-table} for a proof of Table \ref{table:1}. We refer to \defref{def-jordan} and \defref{def-special-matrix-omega} for the notations used in Table \ref{table:1}. Note that for a reversible element, the reversing symmetry group is a degree two extension of the centralizer,  and hence Table \ref{table:1} along with the Jordan decompositions provide us the structure of the reversing symmetry group for the reversible elements in ${\rm SL}(n,\H)$.

\begin{table}[H]	
	\centering{
		\caption{Reversing symmetries for Jordan forms in $ {\SL}(n,\H)$}
		\begin{tabular}{|p{1.5cm}|p{7cm}|p{6cm}|}
			\hline
			\vspace{0.05cm} 	\textbf{ Sr No.  } &\vspace{0.05cm}  Jordan form  (A)&  \vspace{0.05cm}  Reversing element (g) \\
			\hline
			\vspace{0.05cm} \hspace{.44cm}	1 & \vspace{0.005cm} 	$\mathrm{J}(\mu,  n)$,  $\mu \in \{ \pm 1 \}$.  &\vspace{0.005cm}  $\Omega(\mu,  n)$ \vspace{0.15cm} \\
			\hline
			\vspace{.33cm}	\hspace{.55cm}2 &    \vspace{0.005cm}  $ \left( \begin{array}{cc}    \mathrm{J}(\lambda, n)  &  \\  & \mathrm{J}(\lambda^{-1}, n)  \end{array}\right) $,   \vspace{0.15cm} 	\newline $ \lambda \in \C \setminus \{ 0\} $, $\mathrm{Im}(\lambda) \geq 0$, $ |\lambda| \neq 1$.  \vspace{0.005cm} &  \vspace{.005cm} $ \left( \begin{array}{cc}      & \Omega(\lambda,  n) \\
				\Big(	\Omega(\lambda,  n) \Big)^{-1} &   \end{array}\right) $ \vspace{.2cm} \\ 
			\hline
			\vspace{0.05cm} \hspace{.44cm}	3 &\vspace{0.005cm} 	$\mathrm{J}(\alpha,   n)$,  $ \alpha \in \C , \mathrm{Im}(\alpha) >0$,  $ |\alpha|= 1 $.  &\vspace{0.005cm}  $\Omega(\alpha,  n) \, \jb$ \vspace{0.15cm} \\
			\hline
			\vspace{0.05cm} \hspace{.44cm}	4 & {  \vspace{0.0001cm} } $  \left( \begin{array}{cc}    \mathrm{J}(\alpha,   n)  &  \\  & \mathrm{J}(\alpha,  n)  \end{array}\right)$,  \vspace{0.15cm} 
			\newline	$ \alpha \in \C , \mathrm{Im}(\alpha) >0$,  $ |\alpha|= 1 $. \vspace{0.15cm}&   \vspace{0.05cm} $ \left( \begin{array}{cc}      &  \Omega(\alpha,  n) \,  \jb \\
				\Big( \Omega(\alpha,  n) \,  \jb \Big)^{-1} &   \end{array}\right)$  \vspace{0.05cm} \\ 
			\hline
		\end{tabular}
		\label{table:1}}
\end{table}

\textbf{Structure of the paper.}
In Section \ref{sec-prelim}, we establish notations and recall some preliminary results. 
\thmref{thm-strong-rev-SL(n,H)} is proved  in Section \ref{sec-str-rev-SL(n,H)} which classifies the strongly reversible elements in $\mathrm{SL}(n,\mathbb{H})$. 
Section \ref{sec-prod-ske-inv} deals with products of two skew-involutions in $\mathrm{SL}(n,\mathbb{H})$  and we  prove \thmref{thm-rev-SL-prod-inv-PSL}.   Finally, we investigate reversibility in  $\mathrm{PSL}(n,\mathbb{H})$ and prove  \thmref{thm-main-equiv-PSL(n,H)} in Section \ref{sec-rev-PSL(n,H)}.

\section{Preliminaries} \label{sec-prelim}
In this section, we fix some notation and recall some necessary background that will be used throughout this paper. Let $\H:= \R + \R \ib + \R \jb + \R \kb$ denote the division algebra of Hamilton’s quaternions, where $\ib^2=\jb^2=\kb^2=\ib  \jb   \kb = -1$. We consider $\H^n$ as a right $\H$-module. We refer to  \cite{rodman} for an elaborate discussion on the linear algebra over the quaternions.

\begin{definition}\label{def-eigen-M(n,H)}
	Let $A \in  \mathrm{M}(n,\H)$, the algebra of $n \times n$ matrices over $\H$.  A non-zero vector $v \in \H^n $ is said to be a (right) eigenvector of $A$ corresponding to a  (right) eigenvalue  $\lambda \in \H $ if the equality $ A v = v\lambda $ holds.
\end{definition}

Eigenvalues of every matrix over quaternions occur in similarity classes, and each similarity class of eigenvalues contains a unique complex number with non-negative imaginary part. Here, instead of similarity classes of eigenvalues, we will consider the \textit{unique complex representatives} with non-negative imaginary parts as eigenvalues unless specified otherwise. In places where we need to distinguish between the similarity class and a representative, we shall write the similarity class of an eigenvalue representative $\lambda$ by $[\lambda]$.

\begin{definition}[{\cite[p.  94]{rodman}}] \label{def-jordan}
	A {\it Jordan block} $\mathrm{J}(\lambda,m)$ is an $m \times m$ $(m>1)$ matrix  with $ \lambda \in \H$ on the diagonal entries,  $1$ on all of the super-diagonal entries and $0$ elsewhere. For $m=1$,  $\mathrm{J}(\lambda,1) \,:=\, (\lambda)$.
	We will refer to a block diagonal matrix where each block is a Jordan block as  \textit{Jordan form}.  
\end{definition}

\begin{definition}[{\cite[p. 113]{rodman}}] \label{def-det-H}
	Let $A \in \mathrm{M}(n,\H)$. Write $ A = A_1 + A_2 \jb $,  where $ A_1,  A_2 \in \mathrm{M}(n,\C)$.  Consider the embedding $ \Phi:  \mathrm{M}(n,\H)  \longrightarrow  \mathrm{M}(2n,\C)$ defined by
	\begin{equation}\label{eq-embedding-phi}
		\Phi(A) = \begin{psmallmatrix} A_1   &  A_2 \\
			- \overline{A_2} & \overline{A_1}  \\ 
		\end{psmallmatrix}\,, 
	\end{equation}
	where each $ \overline{A_i} $ denotes the complex conjugate of $ A_i$.	
	The determinant of $A \in  \mathrm{M}(n,\H)$ is  defined as 
	$ {\rm det}(A):=  {\rm det}(\Phi(A))$.  In view of the \textit{Skolem-Noether theorem}, the above definition of the quaternionic determinant is independent of the choice of embedding $ \Phi$.  Note that   for $A \in  \mathrm{M}(n,\H)$, ${\rm det(A)}$ is always a non-negative real number; see \cite[Theorem 5.9.2]{rodman}.
\end{definition}

We will consider the Lie groups  $ \mathrm{GL}(n,\H) := \{ g \in \mathrm{M}(n,\H) \mid \det(g) \neq 0 \}$ and  $ \mathrm{SL}(n,\H) := \{ g \in \mathrm{GL}(n,\H) \mid \det(g) = 1 \}$. 
In the following lemma, we recall the Jordan decomposition in  $\mathrm{M}(n,\H)$; see {\cite[Theorem 5.5.3]{rodman} for more details.
	\begin{lemma}[{\cite{rodman}}] \label{lem-Jordan-M(n,H)}
		For every $A \in  \mathrm{M}(n,\H)$,  there exists an invertible matrix $S \in  \mathrm{GL}(n,\H)$ such that 
		\begin{equation} \label{equ-Jordan-M(n,H)}
			SAS^{-1} =  \mathrm{J}(\lambda_1, m_1) \oplus  \cdots \oplus  \mathrm{J}(\lambda_k, m_k),
		\end{equation}
		where $ \lambda_1,  \dots,  \lambda_k $ are complex numbers (not necessarily distinct) and have non-negative imaginary parts.  
		The form  \eqref{equ-Jordan-M(n,H)} is uniquely determined by $A$ up to a permutation of Jordan blocks.
	\end{lemma}

	\subsection{Weyr canonical form}
	In this section, we will recall notations for partitioning a positive integer $n$ and the notion of  Weyr canonical form in ${\rm M }(n,\C)$ from \cite[Section 2]{GLM2}. 
	\begin{definition}[{\cite{COV}}]\label{def-partition-dual} 
		A \textit{partition} of  a positive integer $n$ is a finite sequence $(n_1,n_2,\dots, n_r)$ of  positive integers  such that  $n_1 + n_2 + \dots + n_r =n$ and $n_1 \geq n_2 \geq \dots \geq n_r \geq  1$. Moreover, the \textit{conjugate partition} (or \textit{dual partition})  of  the partition $(n_1,n_2,\dots, n_r)$ of $n$ is the partition $(m_1,m_2,\dots, m_{n_1})$, where $m_j = | \{i \mid n_i \geq j\}|$, cardinality of the set $\{i \mid n_i \geq j\}$.
	\end{definition}
	
	\begin{definition}[{\cite[Section 3.3]{GM}}] \label{def-special-partition-1}
		A \textit{partition }of a positive integer $n$ is an object of the form $$ {\d}(n) := [ d_1^{t_{d_1}}, \dots,   d_s^{t_{d_s}} ],$$ where $t_{d_i}, d_i \in \N, 1\leq i\leq s, $ such that $ \sum_{i=1}^{s} t_{d_i} d_i = n, t_{d_i} \geq 1 $ and  $ d_1 >  \cdots > d_s > 0$. Moreover,  for a partition $\d (n)\, =\, [ d_1^{t_{d_1}},\, \ldots ,\, d_s^{t_{d_s}} ]$ of $n$, define ${\N}_{\d(n)}  \,:=\, \{ d_i \,\mid\, 1 \,\leq\, i \,\leq\, s \}$. Further, define
		$$ {\E}_{\d(n)}  \,:=\, {\N}_{\d(n)}  \cap (2\N), {\O}_{\d(n)}  \,:=\, {\N}_{\d(n)}\setminus {\E}_{\d(n)}, \hbox{ and } \E_{{\d}(n)}^2 := \{ \eta \in \E_{{\d}(n)} \mid \eta \equiv 2 \pmod  {4} \}.$$
	\end{definition}
	
	Here, we introduced two notations $(n_1,n_2,\cdots, n_r)$ and $ {\d}(n)$ for the partition of a positive integer $n$; see  \defref{def-partition-dual} and \defref{def-special-partition-1}. 
	Recall the following lemma from \cite{GLM2}, which gave the relationship between the partition $ {\d}(n)$ and its conjugate partition $ \overline{\d}(n)$. 
	\begin{lemma}[{\cite[Lemma 2.2]{GLM2}}]\label{lem-relation-both-partition}
		Let  $ {\d}(n) = [ d_1^{t_{d_1}}, \dots,   d_s^{t_{d_s}} ]$ be a partition of a positive integer $n$. Then the conjugate partition $ \overline{\d}(n)$ of $ {\d}(n)$ has the following form:
		$$ \overline{\d}(n) = \Big[ (t_{d_1}+t_{d_2}+\cdots+t_{d_s})^{d_s},  (t_{d_1}+t_{d_2}+\cdots+t_{d_{s-1}})^{d_{s-1}-d_s}, \dots,  (t_{d_1}+t_{d_2})^{d_2 -d_3},(t_{d_1})^{d_1 -d_2}\Big].$$
	\end{lemma}

	Now, we will recall the notion of  Weyr canonical form in $\mathrm{M}(n,\C)$ from \cite[Section 2.4]{GLM2}. We refer to  \cite{Sh, COV} for an elaborate discussion on  Weyr canonical forms.
	
	\begin{definition} [{\cite[Definition 2.1.1]{COV}}]\label{def-basic-Weyr-block}
		A \textit{basic Weyr matrix with eigenvalue} $\lambda$  is a matrix $W	 \in  \mathrm{M}(n,\C)$ of the following form: 
		there is a partition $(n_1,n_2, \dots,n_r)$ of $n$ such that, when $W$ is viewed as an $r \times r$ blocked matrices $(W_{ij})_{1 \leq i,j \leq r}$, where the $(i,j)$-th block $W_{ij}$ is an $ n_i \times n_j$ matrix, the following three features are present.
		\begin{enumerate}
			\item The main diagonal blocks  $W_{i,i}$  are the $ n_i \times n_i$ scalar matrices $\lambda \mathrm{I}_{n_i} $ for  $ i = 1,2, \dots, r$.
			\item The first super-diagonal blocks $W_{i,{i+1}}$ are the $ n_i \times n_{i+1}$ matrices in reduced row-echelon form (that is, an identity matrix followed by zero rows) for $ i = 1,2, \dots, r-1$. In other words, the rank of the reduced row-echelon matrix $W_{i,{i+1}}$  is $n_{i+1}$  for $ i = 1,2, \dots, r-1$.
			\item All other blocks of $W$ are zero (that is, $W_{ij}=0$ when $j \neq i, i+1)$.
		\end{enumerate}
		In this case, we say that $W$ has \textit{Weyr structure} $(n_1,n_2, \dots,n_r)$.
		\qed
	\end{definition}

	\begin{definition}[{\cite[Definition 2.1.5]{COV}}]\label{def-general-Weyr-matrix}
		Let $W	 \in  \mathrm{M}(n,\C)$, and let $\lambda_1, \lambda_2,\dots, \lambda_k$ be the distinct eigenvalues of $W$. We say that $W$ is in Weyr form (or is a Weyr matrix) if W is a direct sum of basic Weyr matrices, one for each distinct eigenvalue. In other words, $W$ has the following  form:
		$$W = W_1 \oplus W_2 \oplus \dots \oplus W_k,$$
		where $W_i $ is a basic Weyr matrix with eigenvalue  $\lambda_{i}$ for $i = 1, 2, \dots,k$.
	\end{definition}

	\begin{remark}\label{rem-duality-forms} Note that a permutation transformation conjugates the Jordan and Weyr forms of a square matrix with complex entries. Moreover, for a given $A \in  \mathrm{M}(n,\C)$ with a single eigenvalue $\lambda \in \C$, there is a duality between the partitions corresponding to the Jordan and Weyr form of $A$, respectively. In particular, if ${\d}(n)$ is a partition corresponding to Jordan form of $A \in \mathrm{M}(n,\C)$, then the corresponding Weyr  form $A_W $ of $A$ is  represented  by  the conjugate  partition $\overline{\d}(n)$, as given by  \lemref{lem-relation-both-partition}; see \cite[Theorem 2.4.1]{COV} and \cite[Section 2.5]{GLM2}.
	\end{remark}

	In the following result, we recall the centralizer of a basic Weyr matrix in $\mathrm{M}(n,\C)$. 
	\begin{proposition}[{\cite[Proposition 2.3.3]{COV}}]\label{prop-centralizer-basic-Weyr-block}
		Let $W	 \in  \mathrm{M}(n,\C)$ be an $n \times n$ basic Weyr matrix with the Weyr structure $(n_1,\dots, n_r)$,  $ r \geq 2$. Let $K$ be an $n \times n$ matrix, blocked according to the partition $(n_1,\dots, n_r)$, and let $K_{i,j}$ denote its $(i, j)$-th block (an $n_i \times n_j$ matrix) for $i, j \in \{1,\dots, r\} $. Then $W$ and $K$ commute if and only if $K$ is a block upper triangular matrix for which
		$$ K_{i,j}=  \begin{psmallmatrix}
			K_{i+1,j+1} & 
			\ast  \\
			0	&  \ast
		\end{psmallmatrix} \hbox{ for all } 1 \leq i \leq j \leq r-1.$$
		Here, we have written $K_{i,j}$  as a blocked matrix where the zero block is $(n_i - n_{i+1}) \times n_{j+1}$. The asterisk entries $(\ast)$ indicate that there are no restrictions on the entries in that part of the matrix. The column of asterisks disappears if $n_j = n_{j+1}$, and the $\begin{psmallmatrix}
			0	&  \ast
		\end{psmallmatrix}$ row disappears if $n_i = n_{i+1}$.
	\end{proposition}
	
	We refer to \cite[Example 2.11]{GLM2} for illustration of  \propref{prop-centralizer-basic-Weyr-block}, which describes the centralizer of a Weyr matrix. In this article, for $A \in \mathrm{M}(n,\C)$, we will use the notation $A_W$  to denote the corresponding Weyr form.
	
	In \cite{COV}, the Weyr canonical forms of square matrices over algebraically closed fields are studied. Here, we will extend the notion of Weyr canonical form for matrices over the quaternions. In particular, in view of \lemref{lem-Jordan-M(n,H)}, we can write $A\in \mathrm{M}(n,\H)$ in the Weyr canonical form over the complex numbers.

	\subsection{Matrices commuting with Jordan blocks}
	In the following lemma, we recall the well-known Sylvester’s theorem on solutions to the matrix equation $ AX= XB$; see \cite{rodman} for more details. 
	
	\begin{lemma}[{\cite[Theorem 5.11.1]{rodman}}] \label{lem-Sylvester}
		Let $A\in  \mathrm{M}(n, \H)$  and $B\in  \mathrm{M}(m, \H)$. Then the equation
		$$ AX= XB$$
		has only the trivial solution if and only if $A$ and $B$ have no common eigenvalues.
	\end{lemma}
	The following lemma is useful in understanding the centralizer of a matrix in $\mathrm{M}(n, \H)$.
	\begin{lemma}[{\cite[Proposition 3.1.1]{COV}}]\label{lem-commuting-block-matrices}
		Let $A =   \begin{psmallmatrix}
			A_1   &   \\
			&   A_2  \\
		\end{psmallmatrix}  \in \mathrm{M}(n,\H)$, where $A_1 \in \mathrm{M}(m,\H)$   (resp. $A_2  \in \mathrm{M}(n-m,\H)$) has a single    eigenvalue $\lambda_1$  (resp. $\lambda_2$) such that $[\lambda_1] \neq [\lambda_2]$. If $B  \in \mathrm{M}(n,\H)$ such that $B A = AB$, then $B$ has the following form
		$$B =   \begin{psmallmatrix}
			B_1   &   \\
			&   B_2  \\
		\end{psmallmatrix}, $$ where 
		$B_1 \in \mathrm{M}(m,\H)$ and $B_2   \in \mathrm{M}(n-m,\H)$ such that  $$B_1 A_1 = A_1B_1 \, \hbox{  and  }  \, B_2 A_2 = A_2B_2.$$
	\end{lemma}
	
	\begin{proof}
		Let $B =	\begin{psmallmatrix}
			B_{1,1}  & 	B_{1,2}  \\
			B_{2,1} 	&	B_{2,2} 
		\end{psmallmatrix} \in \mathrm{M}(n,\H)$ be an element  having the same block structure as $A$ such that $BA=AB$. Then we have
		$$	\begin{psmallmatrix}
			B_{1,1}  & 	B_{1,2}  \\
			B_{2,1} 	&	B_{2,2} 
		\end{psmallmatrix} 	\begin{psmallmatrix}
			A_1  &  \\
			&A_2
		\end{psmallmatrix} =	\begin{psmallmatrix}
			A_1  &  \\
			&A_2
		\end{psmallmatrix}\begin{psmallmatrix}
			B_{1,1}  & 	B_{1,2}  \\
			B_{2,1} 	&	B_{2,2} 
		\end{psmallmatrix} . $$
		This implies that
		$$  \begin{psmallmatrix}
			B_{1,1} A_1 & 	B_{1,2} A_2 \\
			B_{2,1} A_1  &	B_{2,2} A_2 
		\end{psmallmatrix} = \begin{psmallmatrix}
			A_1 B_{1,1} & 	 A_1 B_{1,2}\\
			A_2	B_{2,1} &	 A_2 B_{2,2}
		\end{psmallmatrix} 
		.$$	
		Thus, we have $ B_{1,2} A_2 = A_1 B_{1,2} $ and $B_{2,1} A_1 = A_2	B_{2,1}$. Since $A_1$ and $A_2$ have no common eigenvalue, using \lemref{lem-Sylvester}, we conclude that $B_{1,2}$ and $B_{2,1}$ are zero matrices. Let $B_1 = B_{1,1} $ and $B_2 = B_{2,2} $. Then, we have $B_1 A_1 = A_1B_1 \, \hbox{  and  }  \, B_2 A_2 = A_2B_2.$ This proves the lemma.
	\end{proof}
	
	To formulate results related to matrices commuting with Jordan blocks in $\mathrm{M}(n, \H)$, we need to recall the notation for upper triangular \textit{Toeplitz} matrices.  
	
	\begin{definition}[{\cite[p. 95]{rodman}}]\label{def-toeplitz}
		For $\mathbf{x}:=(x_{1},x_{2},\dots, x_{n}) \in \H^{n}$,  we define $\mathrm{Toep}_{n}(\textbf{x}) \in \mathrm{M}(n, \H)$ as \begin{equation}\label{eq-Toeplitz-matrix}
			\mathrm{Toep}_{n}(\mathbf{x}):= \begin{psmallmatrix}
				x_{1} & x_{2}   & \cdots & \cdots & \cdots & x_{n-1} & x_{n}\\
				& x_{1} & x_{2}   & \cdots  & \cdots & x_{n-2}& x_{n-1}\\
				&  & x_{1} & x_{2}  &  \cdots  &  x_{n-3}& x_{n-2}\\
				&  &  &\ddots   & \ddots  &&  \vdots\\
				& & & & \ddots & \ddots   & \vdots\\
				& & & &  & x_{1} & x_{2} \\
				&   &  &  &  &  & x_{1}
			\end{psmallmatrix}.
		\end{equation}
	\end{definition}
	We can also write $\mathrm{Toep}_{n}(\textbf{x})$ as
	\begin{equation}\label{eq-expression-toeplitz}
		\mathrm{Toep}_{n}(\textbf{x}):= [x_{i,j}]_{n \times n} =  \begin{cases}
			0 & \text{if $i>j$ }\\
			x_{j-i+1}& \text{if $i \leq j$}\\
		\end{cases}, \hbox{ where } 1 \leq i,j \leq n.
	\end{equation}
	
	In the following lemma, we recall a basic result that gives the centralizer of the Jordan blocks in $\mathrm{M}(n, \H)$; see  \cite{rodman}  for more details.
	\begin{lemma}[{\cite[Proposition 5.4.2]{rodman}}] \label{lem-commute-Jordan-blocks}
		Let $B \in  \mathrm{M}(n, \H)$ be  such that $B  \,  \mathrm{J}(\lambda,n)=  \mathrm{J}(\lambda,n) \,B$. Then $B$ has the following form:
		$$B =  \mathrm{Toep}_{n}(\mathbf{x}) \hbox{ for some }\mathbf{x}  \in 
		\begin{cases}
			\H^{n}	 & \text{ if }\lambda \in \R ,\\
			\C^{n}  & \text{ if } \lambda \in \C \setminus \R .
		\end{cases}$$
		In particular, if $B \in  \mathrm{M}(n, \H)$  such that $B  \,  \mathrm{J}(e^{\ib \theta},n)=  \mathrm{J}(e^{\ib \theta},n) \,B$, where  $\theta \in (0, \pi)$, then $B =  \mathrm{Toep}_{n}(\mathbf{x}) $ for some $\mathbf{x}  \in \C^{n}$.
	\end{lemma}

	\section{Strongly reversible elements  in  ${\rm SL }(n,\H)$} \label{sec-str-rev-SL(n,H)}
	In this section, we will classify the strongly reversible elements in $\mathrm{SL}(n,\H)$.
	First, we recall the following notation from  \cite{GLM2}, which will be used to investigate reversing symmetries of reversible elements in $\mathrm{SL}(n,\H)$.
	
	\begin{definition} [{\cite[Definition 1.3]{GLM2}}]\label{def-special-matrix-omega}
		For a non-zero  $\lambda \in \C$, define the upper triangular matrix 
		$\Omega( \lambda, n) := [ x_{i,j} ]_{ n \times n}  \in   \mathrm{GL}(n,\C)$ such that:
		\begin{enumerate}
			\item $x_{i,j} = 0$   for all $1\leq i,j \leq n$ such that $i >j$,
			\item  $x_{i,n} = 0 $   for all $1\leq i \leq n-1$,
			\item  $x_{n,n} = 1$,
			\item  For $ 1\leq i \leq j \leq n-1 $,  define 
			\begin{equation}\label{eq-special-matrix-omega}
				x_{i,j} =   - \lambda^{-1} x_{i+1,j} - \lambda^{-2} x_{i+1,j+1}.
			\end{equation}
		\end{enumerate}
	\end{definition}
	Further, recall the following result from \cite[Section 3]{GLM2} which gives relationship between $\Omega( \lambda, n)$, $\Omega( \lambda^{-1}, n)$ and $\mathrm{J}(\lambda, n)$, where  $\lambda \neq 0$. 
	
	\begin{lemma}[{\cite [Lemma 3.2, Lemma 3.3]{GLM2}}] \label{lem-rel-omega-complex} 
		Let  $\Omega( \lambda, n) \in {\rm GL }(n,\C)$ be  as defined in \defref{def-special-matrix-omega}. Then, the following statements are true.
		\begin{enumerate}
			\item \label{cond-1-rel-omega-complex}	 $ \Omega( \lambda, n) \, \mathrm{J}(\lambda^{-1}, n) = \Big(\mathrm{J}(\lambda, n)\Big)^{-1} \,  \Omega( \lambda, n)$.
			\item \label{cond-2-rel-omega-complex}	$\Big(\Omega( \lambda, n)\Big)^{-1} = \Omega( \lambda^{-1}, n)$.
		\end{enumerate}
	\end{lemma}

	In $\mathrm{SL}(n,\H)$, there exist Jordan blocks that are reversible but not strongly reversible. For example, the element $ \mathrm{J}(e^{\ib \theta},   1)  = (e^{\ib \theta}) \in {\rm SL }(1,\H)$,  where  $ \theta \in (0,\pi)$ is reversible, since $\jb e^{\ib \theta} = e^{-\ib \theta} \jb$.
	Now if $ (e^{\ib \theta})$ is strongly reversible, then  $a e^{\ib \theta} = e^{-\ib \theta}a$ form some  $(a) \in  {\rm SL }(1,\H)$ with $a^2=1$. Writing $a = x+ y \jb \in \H$,  $ x, y \in \C$,  it follows that $a= y \jb$. Using the condition $a^2=1$, we get $-|y|^2 =1$, which is a contradiction. Thus, $(e^{\ib \theta})$ is reversible but not strongly reversible in ${\rm SL }(1,\H)$.

	We generalize the above example  for an arbitrary $n$ as follows:
	
	\begin{lemma} \label{lem-rev-jodan-unit-modulus-GL(n, H)}
		Let  $A:= \mathrm{J}(\alpha,   n)$ be the Jordan block in $  {\rm SL }(n,\H)$,  where $ \alpha \in \C \setminus \{- 1, +1\}$  with non-negative imaginary part such that $|\alpha|= 1$.  Then $A$ is reversible but not strongly reversible in  ${\rm SL }(n,\H)$.
	\end{lemma}
	\begin{proof} 	Write $A = \mathrm{J}(e^{ \ib \theta}, n) $,  where  $ \theta \in (0,\pi)$. 
		Since  $\jb z  = \bar{z} \jb$ for all $z \in \C$, we have $ \jb \, P = \overline{P} \,  \jb $ for all $P \in  {\rm GL }(n,\C)$, where $\overline{P}$ is the matrix obtained by taking conjugate of each entry of complex matrix $P$. This implies 
		\begin{equation}\label{eq-quaternion-matrix-conj-j}
			\jb \,  \Big( \mathrm{J}(e^{ \ib \theta}, n) \Big) \,  \jb^{-1} = \overline{ \mathrm{J}(e^{ \ib \theta},  n) }=  \mathrm{J}(e^{ -\ib \theta},  n).
		\end{equation}
		Consider   $\Omega(e^{ \ib \theta} , n) \, \jb  \in   \mathrm{SL}(n,\H)$, where  $\Omega(e^{ \ib \theta}, n) \in  \mathrm{GL}(n,\C)$ is as defined in \defref{def-special-matrix-omega}.   Then,  using \lemref{lem-rel-omega-complex}, we have 
		\begin{equation*}\label{eq-lemma-3.2}
			\Big(\Omega( e^{\ib \theta}, n) \, \jb \Big) A \Big(\Omega( e^{\ib \theta}, n) \, \jb \Big)^{-1}= A^{-1}.
		\end{equation*}
		Therefore, $A$ is reversible in $\mathrm{SL}(n,\H)$.
		
		Suppose that $A$ is strongly reversible in $\mathrm{SL}(n,\H)$. Then there exists an involution $g =[g_{i,j}]_{1\leq i,j \leq n} \in \mathrm{SL}(n,\H)$ such that $gAg^{-1}= A^{-1}$. 
		Note that the set of reversing elements of $A$ is a right coset of the centralizer of $A$; see \cite[Proposition 2.8]{OS}. Therefore, we have
		$$g  = f  \Omega( e^{\ib \theta}, n) \, \jb,$$ 
		where $f \in  \mathrm{GL}(n,\H)$ such that $fA=Af$. 
		Now,  using \lemref{lem-commute-Jordan-blocks}, we have 
		$$f =  \mathrm{Toep}_{n}(\textbf{x}),$$
		where  $\mathbf{x}:=(x_{1},x_{2},\dots, x_{n}) \in \C^{n}$ and  $\mathrm{Toep}_{n}(\textbf{x})  \in \mathrm{GL}(n, \C)$ is as defined in \defref{def-toeplitz}.   This implies
		$$g  =   \mathrm{Toep}_{n}(\textbf{x})  \Omega( e^{\ib \theta}, n) \, \jb.$$ 
		Thus, we get $g_{1,1} =z \jb$, where $z =  (-1)^{n-1}  e^{-2\ib (n-1) \theta} x_1  \in \C$.
		Since $g$ is an involution with an upper triangular form, we have $(g_{1,1})^2 = 1$. Therefore, $|z|^2 =  -1$, which is a contradiction. Hence, $A$ is not strongly reversible in $\mathrm{SL}(n,\H)$. This completes the proof.
	\end{proof}
	
	Although for every $\theta \in (0,\pi)$, the Jordan block $\mathrm{J}(e^{\ib \theta},n)$ is not strongly reversible in $\mathrm{SL}(n,\H)$, Jordan forms in $\mathrm{SL}(n,\H)$ containing such Jordan blocks may still be strongly reversible.

	\begin{lemma} \label{lem-str-rev-jodan-unit-modulus}
		Let  $A:= \mathrm{J}(\alpha, n) \oplus   \mathrm{J}(\alpha,  n)$ be the Jordan form in ${\rm SL }(2n,\H)$,  where $ \alpha \in \C \setminus \{-1, +1\}$  with non-negative imaginary part such that $|\alpha|= 1$.  Then $A$ is strongly reversible in  ${\rm SL }(2n,\H)$.
	\end{lemma}
	
	\begin{proof} 
		Write $A =    \begin{psmallmatrix}
			\mathrm{J}(e^{ \ib \theta},  n) & \\
			&    \mathrm{J}(e^{ \ib \theta},    n) 
		\end{psmallmatrix}$,  where  $ \theta \in (0,\pi)$. 
		Consider $  \Omega(e^{ \ib \theta},  n) \in {\GL}(n,\C)$  as defined in \defref{def-special-matrix-omega}. Let $g =    \begin{psmallmatrix}
			&  \Omega(e^{\ib \theta}, n) \, \jb \\
			\Big( \Omega(e^{\ib \theta},  n)  \, \jb \Big)^{-1} &  
		\end{psmallmatrix}$.  Note that $gAg^{-1} = A^{-1}$ if and only if $  \Big( \Omega( e^{ \ib \theta}, n) \, \jb \Big)  \,  \Big( \mathrm{J}(e^{ \ib \theta}, n) \Big) \,    \Big( \Omega( e^{ \ib \theta}, n) \,  \jb \Big)^{-1} = \Big(\mathrm{J}(e^{ \ib \theta}, n)\Big)^{-1}$. Moreover, $g$ is an involution in $ {\SL}(2n,\H)$. The proof now follows from   \lemref{lem-rel-omega-complex}  and Equation \eqref{eq-quaternion-matrix-conj-j}.
	\end{proof}

	In the following result, we provide the necessary and sufficient conditions for the strong reversibility of elements of $ \mathrm{SL}(n, \mathbb{H})$ that have only the eigenvalue $ \alpha \in \C \setminus \{\pm 1\}$ with a non-negative imaginary part, where $|\alpha|= 1$. We used a similar line of argument as in the proof of \cite[Lemma 4.5]{GLM2}, which is based on the notion of the Weyr canonical form and the structure of the set of reversing symmetries.
	
	\begin{proposition}\label{prop-main-str-rev-quaternion-unit-even}
		Let $A \in \mathrm{SL}(n, \mathbb{H})$ be an element having only eigenvalue $e^{\ib \theta}$, where $\theta \in (0,\pi)$, and $\mathbf{d}(n) = [d_1^{t_{d_1}}, \dots, d_s^{t_{d_s}}]$ be  the corresponding partition in the Jordan decomposition of $A$. Then  $A$ is strongly reversible in $\mathrm{SL}(n, \mathbb{H})$ if and only if  $t_{d_\ell}$ is even for all $1 \leq \ell \leq s$.
	\end{proposition}
	
	\begin{proof} ($\Rightarrow$) In view of  \lemref{lem-Jordan-M(n,H)}, up to conjugacy, we can assume $A$ in the Jordan form with complex entries given by Equation \eqref{equ-Jordan-M(n,H)}. 
		Let $A_{W}$ denote the Weyr form of $A \in \mathrm{GL}(n,\C)$. Then $A_{W} = \delta A \delta^{-1}$ for some $\delta \in \mathrm{GL}(n,\C)$.  
		Using   \lemref{lem-relation-both-partition} and \remref{rem-duality-forms}, the partition $\overline{{\d}}(n) $ representing the Weyr form $A_{W}$ of $A$ is given by:
		\small	$$ \overline{\d}(n) = [ (t_{d_1}+t_{d_2}+\cdots+t_{d_s})^{d_s},  (t_{d_1}+t_{d_2}+\cdots+t_{d_{s-1}})^{d_{s-1}-d_s}, \cdots,  (t_{d_1}+t_{d_2})^{d_2 -d_3},(t_{d_1})^{d_1 -d_2}].$$
		Therefore,  $A_{W}  \in \mathrm{GL}(n,\C)$ is a block matrix with $(d_1)^2$ many blocks and for all $1\leq i \leq d_1$, the size $n_i$ of $i$-th diagonal block is given by
		$$ n_i =  \begin{cases}
			t_{d_1}+t_{d_2}+\cdots+t_{d_s}  & \text{if $1 \leq i \leq d_s$,}\\
			t_{d_1}+t_{d_2}+\cdots+t_{d_{s-r-1} }& \text{if ${d_{s-r}} +1  \leq i \leq {d_{s-r-1}}$},  \text{where $0 \leq r \leq s-2$.}
		\end{cases}$$ 
		This implies that $(i,j)$-th block of $A_{W}$  has the size $ n_i \times  n_j$; see \cite[Section 2.2]{GLM2}. Further, $(i,j)$-th block of  the Weyr form $A_{W}$ and its inverse $(A_{W} )^{-1}$  can be given as follows:
		$$ (A_{W})_{i,j} =  \begin{cases}
			e^{\ib \theta}	\mathrm{I}_{n_i} & \text{if $j=i$ }\\
			\mathrm{I}_{n_i \times n_{i +1}} & \text{if $j =i+1$}\\
			\mathrm{O}_{n_i \times n_{j}}  & \text{otherwise}
		\end{cases}, \hbox{ and } $$
		$$((A_{W} )^{-1})_{i,j} =  \begin{cases}
			e^{-\ib \theta}\mathrm{I}_{n_i}  & \text{if $j=i$ }\\
			(-1)^k  	e^{-(k+1)\ib \theta} \mathrm{I}_{n_i \times n_{i +k}} & \text{if $j =i+k$},  1 \leq k \leq d_1 - i\\
			\mathrm{O}_{n_i \times n_{j}}  & \text{otherwise}
		\end{cases},$$
		where $1\leq i,  j \leq d_1$ and $\mathrm{O}_{n_i \times n_{j}}$ denotes the $n_i \times n_{j}$ zero matrix.
		Consider $\Omega_{W} := [ X_{i,j} ]_{d_1 \times d_1}  \in   \mathrm{GL}(n,\C)$  such that
		\begin{equation}
			X_{i,j} =\begin{cases}
				\mathrm{O}_{n_i \times n_{j}}  & \text{if $i>j $}\\
				\mathrm{O}_{n_i \times n_{j}}  & \text{if $j =d_1,  i  \neq d_1$}\\
				(-1)^{d_1-i}  \, \Big( e^{-2(d_1-i) \ib \theta}\Big) \,  \mathrm{I}_{n_i}& \text{if $j=i$ }\\
				(-1)^{d_1-i} \,  \binom{d_1-i-1}{j-i}  \,  \Big(e^{-(2d_1+i+j) \ib \theta} \Big) \,  \mathrm{I}_{n_i \times n_{j}}  & \text{if $ j>i, j  \neq d_1$ }
			\end{cases},
		\end{equation}
		where  $\binom{d_1-i-1}{j-i} $ denotes the binomial coefficient.  Let $\tau = \Omega_{W}  \, \jb  \in  \mathrm{GL}(n,\H)$. 
		Then by using a similar argument as used in \lemref{lem-rev-jodan-unit-modulus-GL(n, H)} and \cite[Lemma 3.2]{GLM2}, we have 
		\begin{equation}\label{eq-rev-Weyr-unit-H}
			\tau A_{W} \tau^{-1}=(A_{W} )^{-1}.
		\end{equation}
		Let $f=[P_{i,j}]_{d_1 \times d_1} \in \mathrm{GL}(n,\H)$ be an $n \times n$ matrix  commuting with  Weyr form $A_{W} \in \mathrm{GL}(n,\C)$, where both $f$ and $A_{W}$ are blocked according to the partition $\overline{{\d}}(n) $. Then, using a similar argument as in  \lemref{lem-commute-Jordan-blocks},  we can conclude that  $f$ is a matrix with complex entries. Thus, we have $f \in \mathrm{GL}(n,\C)$ such that $f 	(e^{\ib \theta}	\mathrm{I}_{n}) =  	(e^{\ib \theta}	\mathrm{I}_{n}) f$.  Then $fA_W =A_W f$ implies that $fN =Nf$, where $N= A_W - e^{\ib \theta}	\mathrm{I}_{n}$ is a nilpotent basic Weyr matrix given by the partition $ \overline{\d}(n) $.  Therefore,  using \propref{prop-centralizer-basic-Weyr-block}, we have that $f \in \mathrm{GL}(n,\C)$ is an upper triangular block matrix such that $i$-th diagonal block $P_{i,i}$ of $f$  is given as follows:
		\begin{enumerate} 
			\item if  $1 \leq i \leq d_s$, then $ P_{i,i}=  \begin{psmallmatrix}
				{P}_{1} &\ast  & \ast &  \ast  & \cdots  &  \ast \\
				& {P}_{2} &  \ast   & \cdots  & \cdots  &   \ast \\
				&  &\ddots   & \ddots  && \vdots \\
				& & & \ddots & \ddots   & \vdots \\
				& & &  & {P}_{{s-1}} &  \ast \\
				&  &  &  &  &{P}_{s}
			\end{psmallmatrix}$,
			
			\item if  ${d_{s-r}} +1  \leq i \leq {d_{s-r-1}}$, then $ P_{i,i}=  \begin{psmallmatrix}
				{P}_{1} &\ast  & \ast &  \ast  & \cdots  &  \ast \\
				& {P}_{2} &  \ast   & \cdots  & \cdots  &   \ast \\
				&  &\ddots   & \ddots  && \vdots \\
				& & & \ddots & \ddots   & \vdots \\
				& & &  & {P}_{{s-r-2}} &  \ast \\
				&  &  &  &  &{P}_{s-r-1}
			\end{psmallmatrix},$
			
			\item  if  $d_2 +1  \leq i \leq d_1$, then $P_{i,i} = P_1$,
		\end{enumerate}
		where  $0 \leq r \leq s-3$ and $P_k \in \mathrm{GL}(t_{d_k}, \C)$ for all $ 1 \leq k \leq s$.

		Since $A$ is strongly reversible in $\mathrm{SL}(n, \mathbb{H})$, there exists  an involution $g$ in $\mathrm{SL}(n,\H)$  such that  $gAg^{-1} =A^{-1}$. Consider $h = [Y_{i,j}]_{d_1 \times d_1}:= \delta g \delta ^{-1}$ in $\mathrm{SL}(n,\H)$. Then $h$ is an involution in $\mathrm{SL}(n,\H)$ such that 
		\begin{equation}\label{eq-relation-reverser-Weyr-Jordan-unipotent-H}
			h A_{W} h^{-1} =(A_{W})^{-1}.
		\end{equation}
		Note that the set of reversing elements of $A_W$ is a right coset of the centralizer of $A_W$. Therefore, using Equations   \eqref{eq-rev-Weyr-unit-H} and \eqref{eq-relation-reverser-Weyr-Jordan-unipotent-H}, we have
		$$h =  f  \tau.$$
		This implies that $h \in \mathrm{GL}(n,\H)$ is an upper triangular block matrix, and the first diagonal block of $h$ is given by  $$Y_{1, 1} =   \begin{psmallmatrix}
			{Q}_{1} \,  \jb &\ast  & \ast &  \ast  & \cdots  &  \ast \\
			& {Q}_{2}  \, \jb  &  \ast   & \cdots  & \cdots  &   \ast \\
			&  &\ddots   & \ddots  && \vdots \\
			& & & \ddots & \ddots   & \vdots \\
			& & &  & {Q}_{{s-1}} \,   \jb &  \ast \\
			&  &  &  &  &{Q}_{s} \, \jb 
		\end{psmallmatrix},$$
		where $	Q_k= 	(-1)^{d_1-1}  \, \Big( e^{-2(d_1-1) \ib \theta} \Big) P_{k}   \in \mathrm{GL}(t_{d_k}, \C)$ for all $1 \leq k \leq s$. Since $h$ is an involution, we have $Y_{1, 1}$, and hence $Q_k \jb$ is an involution for all $1 \leq k \leq s$.  Then we have
		$$ Q_k \overline{Q_k}= -\mathrm{I}_{t_{d_k}}.$$
		This implies  that $|\mathrm{det}( Q_k)|^2 = (-1)^{t_{d_k}}$. Therefore, $t_{d_k}$ is even for all $1 \leq k \leq s$. 
		
		($\Leftarrow$) The proof of the converse part of the lemma follows from \lemref{lem-str-rev-jodan-unit-modulus}.
		This completes the proof.
	\end{proof}

	\subsection{Proof of \thmref{thm-strong-rev-SL(n,H)}}
	($\Leftarrow$)  Recall that  if $g \in {\rm GL }(n,\H)$ is an involution, then $ g \in {\rm SL }(n,\H)$; see \defref{def-det-H}.   The proof now follows from \cite[Theorem 5.4]{GLM1}
	
	($\Rightarrow$) Let $A$ be a strongly reversible element in ${\rm SL }(n,\H)$. If $A$ has no non-real eigenvalues of unit modulus, the proof is complete.
	Now, assume that $A$ has a non-real unit modulus eigenvalue $\mu$, with multiplicity $m \geq 1$. If $m=n$, then the proof follows from  \propref{prop-main-str-rev-quaternion-unit-even}. 
	
	Assume that $1\leq m <n$. 
	Write $\mu= e^{\ib \theta_o}$ for some  $\theta_o \in  (0,\pi)$.  
	Using  \lemref{lem-Jordan-M(n,H)}, we can express $A$ as
	$$A =	\begin{psmallmatrix}
		A_1  &  \\
		&A_2
	\end{psmallmatrix}, $$
	where $A_1 \in \mathrm{GL}(m,\H)$  and  $A_2 \in \mathrm{GL}(n-m,\H)$ such that $1\leq m <n$. Moreover, $A_1$ has a single eigenvalue $e^{\ib \theta_o}$, and for each eigenvalue $\lambda$ of $A_2$, we have $[e^{\ib \theta_o}] \neq [\lambda]$. Let $B \in \mathrm{GL}(n,\H)$ such that $BA=AB$. Then \lemref{lem-commuting-block-matrices} implies that $B$ has the following form:
	$$B= \begin{psmallmatrix}
		B_{1}  & 	  \\
		&	B_{2} 
	\end{psmallmatrix}, $$
	where $ B_{1}  \in \mathrm{GL}(m,\H) $ and $B_{2} \in \mathrm{GL}(n-m,\H)$ such that  $B_1A_1=A_1B_1$ and $B_2A_2=A_2B_2$.  
	Since $A$ is reversible, using \thmref{thm-rev-SL(n,H)}, \lemref{lem-rev-jodan-unit-modulus-GL(n, H)}, and Table \ref{table:1}, we can find a  reversing element $h$ for $A$ which has the form
	$$h= \begin{psmallmatrix}
		h_1  & 	  \\
		&	h_2
	\end{psmallmatrix}, $$
	where $ h_1   \in \mathrm{GL}(m,\H) $ and $h_2 \in \mathrm{GL}(n-m,\H)$ such that $h_i A_i(h_i)^{-1}= (A_i)^{-1}$ for  all $ i \in \{1,2\} $.
	
	Let   $g \in \mathrm{SL}(n,\H)$ be an involution  such that $gAg^{-1} = A^{-1}$.  Since the set of reversing elements of $A$ is a right coset of the centralizer of $A$, we have that $g$ has the following form
	$$g= \begin{psmallmatrix}
		g_1  & 	  \\
		&	g_2
	\end{psmallmatrix}, $$
	where $ g_1   \in \mathrm{GL}(m,\H) $ and $g_2 \in \mathrm{GL}(n-m,\H)$  are involutions such that $g_i A_i(g_i)^{-1}= (A_i)^{-1}$ for $ i \in \{1,2\} $. Thus, we have  an involution $ g_1   \in \mathrm{GL}(m,\H) $ such that $g_1A_1(g_1)^{-1}= (A_1)^{-1}$.  The proof of the theorem now follows from \propref{prop-main-str-rev-quaternion-unit-even}.
	\qed

	\subsection{Proof of Table \ref{table:1}} \label{sec-proof-table}
	The proof of Table \ref{table:1}  follows from \cite[Table 1]{GLM2},  \lemref{lem-rev-jodan-unit-modulus-GL(n, H)} and  \lemref{lem-str-rev-jodan-unit-modulus}.
	\qed	
	
	\section{Product of two skew-involutions in ${\rm SL }(n,\H)$}\label{sec-prod-ske-inv}
	In this section, we demonstrate that every reversible element in ${\rm SL }(n,\H)$  can be expressed as a product of two skew-involutions in ${\rm SL }(n,\H)$. In view of  \thmref{thm-rev-SL(n,H)} which classifies the reversible elements in ${\rm SL }(n,\H)$, it is sufficient to consider certain Jordan forms in ${\rm SL }(n,\H)$. Note the following two results, which will be used in proving \thmref{thm-rev-SL-prod-inv-PSL}. 
	
	\begin{lemma}\label{lem-prod-skew-inv-unit-mod}
		Let  $A:=\mathrm{J}(\mu,  n)$ be the Jordan block in $ {\rm SL }(n,\H)$,  where $ \mu \in \C$  with non-negative imaginary part such that $|\mu|= 1$.  
		Then there exists a skew-involution  $g \in \mathrm{SL}(n,\H) $  such that $g A  g^{-1}= A^{-1}$.
	\end{lemma}
	\begin{proof} Write $A:=\mathrm{J}(e^{\ib \theta}, n )$, where $\theta \in [0, \pi]$.
		Let $g:= \Omega(e^{ \ib \theta} , n) \, \jb  \in   \mathrm{GL}(n,\H)$, where  $\Omega(e^{ \ib \theta}, n) \in  \mathrm{GL}(n,\C)$ is as defined in \defref{def-special-matrix-omega}.   Then,  using \lemref{lem-rel-omega-complex} (\ref{cond-1-rel-omega-complex}), we have $gAg^{-1}= A^{-1}$; see also \lemref{lem-rev-jodan-unit-modulus-GL(n, H)}.
		In view of \lemref{lem-rel-omega-complex} (\ref{cond-2-rel-omega-complex}), we have 
		$$\Omega(e^{ \ib \theta} , n) \Omega(e^{- \ib \theta} , n) = \mathrm{I}_{n}.$$
		This implies that $$g^2 =  \Omega(e^{ \ib \theta} , n) \, \jb  \Omega(e^{ \ib \theta} , n) \, \jb = \Omega(e^{ \ib \theta})  \overline{ \Omega(e^{ \ib \theta} , n)} (\jb^2)  = - \Omega(e^{ \ib \theta} , n) \Omega(e^{- \ib \theta} , n) = -\mathrm{I}_{n}.$$
		This proves the lemma.
	\end{proof}
	
	\begin{lemma} \label{lem-prod-skew-inv-non-unit-mod}
		Let  $A:= \mathrm{J}(\lambda, n) \  \oplus \ \mathrm{J}(\lambda^{-1}, n)$  be the Jordan form in $ {\SL}(2n,\H)$,   where  $\lambda  \in \C \setminus \{0\} $  with non-negative imaginary part  such that  $ |\lambda| \neq 1$. Then there exists a skew-involution  $g $ in $\mathrm{SL}(2n,\H) $  such that $g A  g^{-1}= A^{-1}$.
	\end{lemma}
	
	\begin{proof} 
		Write $A= \begin{psmallmatrix}
			\mathrm{J}(\lambda, n) &  \\
			&  \mathrm{J}(\lambda^{-1}, n) 
		\end{psmallmatrix}  \in  {\SL}(2n,\H)$.
		Let $g =    \begin{psmallmatrix}
			& \Omega(\lambda,  n) \\
			-	\Omega(\lambda^{-1},  n) &  
		\end{psmallmatrix},$ where $  \Omega(\lambda,  n) \in {\GL}(n,\C)$  is as defined in \defref{def-special-matrix-omega}.  Then  \lemref{lem-rel-omega-complex} (\ref{cond-2-rel-omega-complex}) implies that $g$ is a skew-involution in $ {\SL}(2n,\H)$.  Moreover, $gAg^{-1} = A^{-1}$ if and only if $  \Omega( \lambda, n) \, \mathrm{J}(\lambda^{-1}, n) = \Big(\mathrm{J}(\lambda, n)\Big)^{-1} \,  \Omega( \lambda, n).$  The proof now follows from \lemref{lem-rel-omega-complex} (\ref{cond-1-rel-omega-complex}).
	\end{proof}

	\subsection{Proof of \thmref{thm-rev-SL-prod-inv-PSL}}
	Note that an element $A \in {\rm SL }(n,\H)$ can be written as a product of two skew-involutions if and only if there exists a skew-involution $g \in \mathrm{SL}(n,\H) $  such that $g A  g^{-1}= A^{-1}$.
	The proof now follows from \thmref{thm-rev-SL(n,H)}, \lemref{lem-Jordan-M(n,H)},  \lemref{lem-prod-skew-inv-unit-mod}, and \lemref{lem-prod-skew-inv-non-unit-mod}
	\qed
	
	\section{Reversibility  in  ${\rm PSL}(n,\H)$} \label{sec-rev-PSL(n,H)}
	In this section, we will investigate the reversibility problem in the quaternionic projective linear group $\mathrm{PSL}(n,\H):= \mathrm{SL}(n,\H)/\{\pm \mathrm{I}_{n}\}$. Recall that 
	for every $g,h \in \mathrm{SL}(n,\H)$,  we have
	$$ [g]=[h]  \Leftrightarrow  g=\pm h.$$
	Moreover, every skew-involution  in $\mathrm{SL}(n,\H)$ is an  involution in $\mathrm{PSL}(n,\H)$. 
	The following result provides the relation between reversibility in the groups $\mathrm{SL}(n,\H)$ and $\mathrm{PSL}(n,\H) $, cf. \cite[Lemma 4.1]{DGL}.
	
	\begin{lemma}\label{lem-rev-PSL-eq-conditions}
		An element $[g] \in \mathrm{PSL}(n,\H)$ is reversible if and only if there exists $h \in \mathrm{SL}(n,\H)$ such that either of the following conditions holds.
		\begin{enumerate}
			\item $hgh^{-1}= g^{-1}$.
			\item $hgh^{-1}= -g^{-1}$.
		\end{enumerate}
	\end{lemma}

	Now, we will investigate the equation $g A  g^{-1}= -A^{-1}$, where $A, g \in \mathrm{SL}(n,\H) $.  In \cite{PaSa}, authors did such investigation for matrices over an arbitrary field $\F$, and they gave necessary and sufficient conditions for $A \in  \mathrm{GL}(n, \F)$ to be a product of an involution and a skew-involution; see \cite[Theorem 5]{PaSa}.
	The following result classify all  elements $A \in \mathrm{SL}(n,\H) $ such that  $A$ is conjugate to $-A^{-1}$ in $ \mathrm{SL}(n,\H) $, cf. \thmref{thm-rev-SL(n,H)}.
	
	\begin{theorem} \label{thm-rev-PSL-type-2} An element $A \in  {\rm SL }(n,\H)$ satisfies 
		$g A  g^{-1}= -A^{-1}$ for some $g \in \mathrm{SL}(n,\H) $ if and only if the Jordan blocks in the Jordan decomposition of $A$ can be partitioned into pairs $ \{ \mathrm{J}(\lambda, s),\mathrm{J}(-\lambda^{-1}, s)\} $ or singletons $\{\mathrm{J}(\ib, t  )\}$,  where $\lambda \in \C \setminus \{0\}$  with non-negative imaginary parts and $\lambda \neq \ib$.  
	\end{theorem}
	
	\begin{proof} 
		Note that  $g A  g^{-1}= -A^{-1}$ for some $g \in \mathrm{SL}(n,\H) $ if and only if $A$ and $-A^{-1}$  have the same Jordan form; see \lemref{lem-Jordan-M(n,H)}. Let $a \in \H \setminus \{0\}$. Recall that  $[a] =[-a^{-1}]$ if and only if  $\mathrm{Re}(a) = \mathrm{Re}(-a^{-1})$ and $|a|= |-a^{-1}|$; see \cite[Theorem 2.2.6 (5)]{rodman}. Now, equation  $|a|= |-a^{-1}|$ implies that $|a|=1$. This implies that $\mathrm{Re}(a) = \mathrm{Re}(-a^{-1}) = \mathrm{Re}(-\overline{a})$. Therefore, we have  $[a] =[a^{-1}]$ if and only if $\mathrm{Re}(a) =0$ and $|a|=1$. Hence, for  $\lambda \in \C \setminus \{0\}$  with non-negative imaginary parts,  we have
		$$[\lambda] =[-\lambda^{-1}]  \Leftrightarrow   \lambda =\ib.$$ The proof  now follows using a similar line of arguments as in \cite[Theorem 5.1]{GLM1}
	\end{proof}

	\begin{lemma}\label{lem-2-str-rev-PSL-type2}
		Let $A:= \mathrm{J}(\lambda, n) \oplus \mathrm{J}(-\lambda^{-1}, n)$ be the Jordan form in $ {\rm SL }(2n,\H)$, where $\lambda \in \C \setminus \{0\}$  with non-negative imaginary parts such that $ \lambda \neq \ib.$ Then there exists an involution  $g \in \mathrm{SL}(2n,\H) $  such that $g A  g^{-1}= -A^{-1}$.
	\end{lemma}
	\begin{proof}
		Write $A= \begin{psmallmatrix}
			\mathrm{J}(\lambda, n) &  \\
			&  \mathrm{J}(-\lambda^{-1}, n) 
		\end{psmallmatrix}  \in  {\SL}(2n,\H)$.
		Since  $	-  \Big(\mathrm{J}(\lambda, n)\Big)^{-1}$  and $\mathrm{J}(-\lambda^{-1}, n)$ are conjugate in  $ {\GL}(n,\C)$, there exists an element $P \in {\GL}(n,\C)$ such that 
		\begin{equation}
			P \,  \mathrm{J}(-\lambda^{-1}, n)  \, P^{-1}= -\Big(\mathrm{J}(\lambda, n)\Big)^{-1}.
		\end{equation}
		Consider $g =    \begin{psmallmatrix}
			& P \\
			P^{-1} &  
		\end{psmallmatrix} \in {\GL}(2n,\C) \subset {\SL}(2n,\H)$.   Then   $g$ is an involution in $ {\SL}(2n,\H)$; see \defref{def-det-H}. Moreover, $gAg^{-1} = -A^{-1}$ if and only if $P \,  \mathrm{J}(-\lambda^{-1}, n)  \, P^{-1}= -\Big(\mathrm{J}(\lambda, n)\Big)^{-1}.$  Hence, the proof follows.
	\end{proof}
	
	Note the following example.
	\begin{example}\label{examp-iota-skew}
		Let $A =  \mathrm{J}(\ib, 5  ) \in {\rm SL }(5,\H)$. Consider  $g= \begin{psmallmatrix}
			1 & -3\ib  & -3  & \ib &0 \\
			& -1 & 2\ib   & 1 & 0\\
			&  & 1 & -\ib  &0 \\
			& &  & -1 &0 \\
			&  &  &  & 1
		\end{psmallmatrix} $  in $ {\rm SL }(5,\H)$. Note that the following equations hold.
		$$gA =  \begin{psmallmatrix}
			1 & -3\ib  & -3  & \ib &0 \\
			& -1 & 2\ib   & 1 & 0\\
			&  & 1 & -\ib  &0 \\
			& &  & -1 &0 \\
			&  &  &  & 1
		\end{psmallmatrix}   \begin{psmallmatrix}
			\ib& 1  & 0  & 0 &0 \\
			& \ib& 1   & 0 & 0\\
			&  & \ib & 1  &0 \\
			& &  & \ib&1 \\
			&  &  &  &\ib
		\end{psmallmatrix}  =  \begin{psmallmatrix}
			\ib& 4 & -6 \ib & 4& \ib \\
			& -\ib& -3   & 3\ib & 1\\
			&  & \ib & 2  &-\ib \\
			& &  &- \ib&-1 \\
			&  &  &  &\ib
		\end{psmallmatrix}, \hbox{ and }   $$
		$$ A^{-1}g =  \begin{psmallmatrix}
			-	\ib& 1  & \ib  & -1 &-\ib \\
			& -	\ib& 1  & \ib  & -1\\
			&  & -	\ib& 1  & \ib \\
			& &  &  -	\ib& 1 \\
			&  &  &  & -\ib
		\end{psmallmatrix}  \begin{psmallmatrix}
			1 & -3\ib  & -3  & \ib &0 \\
			& -1 & 2\ib   & 1 & 0\\
			&  & 1 & -\ib  &0 \\
			& &  & -1 &0 \\
			&  &  &  & 1
		\end{psmallmatrix} =  \begin{psmallmatrix}
			-	\ib& -4 & 6 \ib & -4& -\ib \\
			& \ib& 3   & -3\ib & -1\\
			&  & -\ib & -2  &\ib \\
			& &  &\ib&1 \\
			&  &  &  &-\ib
		\end{psmallmatrix}.$$
		Moreover, $g^2 = \mathrm{I}_{5}.$ Therefore, $g$ is  an involution  in $ \mathrm{SL}(5,\H) $  such that $g A  g^{-1}= -A^{-1}$.
	\end{example}

	In the following lemma, we will generalize 	Example \ref{examp-iota-skew}.
	\begin{lemma}\label{lem-1-str-rev-PSL-type2}
		Let $A:= \mathrm{J}(\ib, n  )$ be the Jordan block in $ {\rm SL }(n,\H)$. Then there exists an involution  $g$ in $\mathrm{SL}(n,\H) $  such that $g A  g^{-1}= -A^{-1}$.
	\end{lemma}
	\begin{proof} Write $ A = [a_{i,j}]_{ n \times n} $ and $ A^{-1} = [b_{i,j}]_{ n \times n} $, where 
		\begin{equation*}
			a_{i,j} =  \begin{cases}
				\ib & \text{if $j=i$ }\\
				1 & \text{if $j =i+1$}\\
				0 & \text{otherwise}
			\end{cases}, \hbox{ and }  b_{i,j} =  \begin{cases}
				-\ib & \text{if $j=i$ }\\
				-  (\ib)^{j-i+1} & \text{if $j>i$}\\
				0 & \text{otherwise.}
			\end{cases}
		\end{equation*}
		
		Consider the upper triangular matrix 
		$g := [ x_{i,j} ]_{ n \times n}  \in   \mathrm{GL}(n,\C)$, where
		\begin{enumerate}
			\item $x_{i,j} = 0$ \quad for  all $1\leq i,j \leq n$ such that $i >j$,
			\item  $x_{i,n} = 0  $  \quad  for all  $1\leq i \leq n-1$,
			\item  $x_{n,n} = 1$,
			\item  for all $ 1\leq i \leq j \leq n-1 $,  define
			\begin{equation}\label{eq-special-matrix-iota-skew}
				x_{i,j} =   \ib x_{i+1,j} - x_{i+1,j+1}.
			\end{equation}
		\end{enumerate}
		For more clarity, we can  write $g = [ x_{i,j} ]_{ 1\leq i,  j \leq n}  \in   \mathrm{GL}(n,\C)$  explicitly as follows:
		\begin{equation}
			x_{i,j} =\begin{cases}
				0 & \text{if $j <i $} \\
				0 & \text{if $j =n,  i  \neq n$}\\
				(-1)^{n-i} & \text{if $j=i$ }\\
				(-1)^{n-i} \,  \binom{n-i-1}{j-i}  \,  (-\ib)^{j-i}& \text{if $ j>i, j  \neq n$ }
			\end{cases},
		\end{equation}
		where  $\binom{n-i-1}{j-i} $ denotes the binomial coefficient.
		Observe that for   all $1\leq i <j\leq n$, we have 
		\begin{equation}\label{eq-equivalent-special-matrix-iota-skew}
			x_{i,j-1}= (i)^2 x_{i+1,j} + (i)^3 x_{i+2,j} +  \dots + (i)^{j-i+1}  x_{j,j}
		\end{equation}
		
		To prove the lemma, it suffices to establish the following two claims:
		\begin{enumerate}
			\item 	\textit{Claim 1. $gA = -A^{-1}g$.}\\
			\textit{Proof.}
			Note that   for all $1\leq i \leq n$,  we have 
			$$(gA)_{i,i} = - ( A^{-1} g)_{i,i} = 	(-1)^{n-i} \ib = \ib x_{i,i} .$$
			Since matrices  under consideration are upper triangular, so it is enough to prove the following equality:
			$$
			(gA)_{i,j} = - ( A^{-1} g)_{i,j} =  x_{i,j-1} + \ib \,  x_{i,j} \hbox { for all } 1 \leq i < j\leq n. 
			$$
			To see this, note that  for $1 \leq i < j\leq n$, we have
			$$	(gA)_{i,j}  = \sum_{k=1}^{n} (x_{i,k}) \,  (a_{k,j}) = \sum_{k=i}^{j} (x_{i,k}) \,  (a_{k,j})= x_{i,j-1} a_{j-1,j} + x_{i,j} a_{j,j}=x_{i,j-1} + \ib \,  x_{i,j}.$$
			Further,  for $1 \leq i < j\leq n$,  we have
			\small	\begin{align*}
				( A^{-1} g)_{i,j} & = \sum_{k=1}^{n} (b_{i,k})  (x_{k,j}) = \sum_{k=i}^{j} (b_{i,k})   (x_{k,j}) = b_{i,i} x_{i,j} + b_{i,i+1} x_{i+1,j}+ \dots + b_{i,j} x_{j,j}
				\\&	=(-\ib x_{i,j}) - \Big((i)^2 x_{i+1,j} + (i)^3 x_{i+2,j} +  \dots + (i)^{j-i+1}  x_{j,j} \Big).
			\end{align*}
			Using  Equation \eqref{eq-equivalent-special-matrix-iota-skew},
			we get
			$$ 
			( A^{-1} g)_{i,j}   = (-\ib x_{i,j}) - (x_{i,j-1})= -(  x_{i,j-1} +\ib x_{i,j}) \hbox { for all } 1 \leq i < j\leq n. 
			$$
			Therefore, we have $gA = -A^{-1}g$.
			
			\item \textit{Claim 2. $g^2= \mathrm{I}_n$.}\\
			\textit{Proof.} Note that  we have $(g^2)_{i,i} =1$ for all $1 \leq i \leq n$. Moreover,  for all $1 \leq i < j \leq n$, we have  $$(g^2)_{i,j} =   \sum_{k=1}^{n} (x_{i,k}) \,  (x_{k,j}) = \sum_{k=i}^{j} (x_{i,k}) \,  (x_{k,j})  .$$ This  implies that  for all $1 \leq i < j \leq n$, we have
			$$(g^2)_{i,j} =    \sum_{k=i}^{j}  (-1)^{n-i}  \,  \binom{n-i-1}{k-i }  \, (-\ib)^{k-i}  \,  (-1)^{n-k} \,    \binom{n-k-1}{j-k }  \, (-\ib)^{j-k}.$$
			Therefore,  for all $1 \leq i < j \leq n$, we have
			\begin{equation}\label{eq-square-special-matrix-iota-skew}
				(g^2)_{i,j} =    (-\ib)^{j-i}  \,   \sum_{k=i}^{j}  (-1)^{(-i-k)}   \,  \binom{n-i-1}{k-i }   \,     \binom{n-k-1}{j-k }.
			\end{equation}
			By substituting $m = k-i$ in  Equation \eqref{eq-square-special-matrix-iota-skew}, we get
			$$
			(g^2)_{i,j} =   (-\ib)^{j-i} \,   \sum_{m=0}^{j-i} (-1)^{(-2i-m)}  \binom{n-i-1}{m } \binom{(n-i-1)-m}{(j-i) -m}.
			$$
			Now, recall the following well-known binomial identities concerning binomial coefficients from \cite[Section 1.2]{burton}:
			\begin{enumerate} 
				\item ~\label{eq-Newton-identity}  
				\textit{Newton's identity:} $\binom{n}{k} \binom{k}{r}= \binom{n}{r} \binom{n-r}{k-r}  $ for all $ 0\leq r \leq k \leq n$.
				\vspace{.1cm}
				\item ~\label{eq-binomial-alternate-sum}
				For $n\geq 1$,   $\sum_{k=0}^{n} (-1)^{k} \binom{n}{k}=0$.
			\end{enumerate}
			In view of the Newton's identity (\ref{eq-Newton-identity}) and identity (\ref{eq-binomial-alternate-sum}), we get
			$$ (g^2)_{i,j}  =  (-\ib)^{j-i}\,  \binom{n-i-1}{j-i } \,  \sum_{r=0}^{j-i} \, (-1)^r \,  \binom{j-i}{r } =0 \hbox { for all } 1\leq i < j \leq n.$$
			Therefore,  $(g^2)= \mathrm{I}_n$  in ${\rm GL }(n,\C)$.   
		\end{enumerate}
		Hence, the proof follows.
	\end{proof}
	
	\begin{theorem}\label{thm-rev-eq-str-rev-PSL-type2}
		Let $A \in \mathrm{SL}(n,\H) $. Then the following conditions are equivalent.
		\begin{enumerate}
			\item $g A  g^{-1}= -A^{-1}$ for some $g \in \mathrm{SL}(n,\H) $.
			\item  $h A  h^{-1}= -A^{-1}$ for some involution $h \in \mathrm{SL}(n,\H) $. 
		\end{enumerate}
	\end{theorem}
	\begin{proof}
		The proof of theorem follows from \thmref{thm-rev-PSL-type-2}, \lemref{lem-2-str-rev-PSL-type2}  and \lemref{lem-1-str-rev-PSL-type2}.
	\end{proof}

	The following result generalizes \cite [Proposition 4.3]{DGL} for an arbitrary $n$.
	\begin{proposition}	\label{prop-rev-type-2-PSL-skew-inv-prod}
		Let $A \in \mathrm{SL}(n,\H) $ be an element such that $g A  g^{-1}= -A^{-1}$ for some $g \in \mathrm{SL}(n,\H) $. Then $A$ can be written as a product of an involution and a skew-involution in $\mathrm{SL}(n,\H) $.
	\end{proposition}
	\begin{proof}
		In view of the \thmref{thm-rev-eq-str-rev-PSL-type2},  there exists an involution $h \in \mathrm{SL}(n,\H) $ such that $hA h^{-1}= -A^{-1}$. This implies that
		$$A = (- h^{-1}  A^{-1})   (h),$$
		where $(- h^{-1}  A^{-1})^2  = -\mathrm{I}_n$ and $h^{2}= \mathrm{I}_n$. This completes the proof.
	\end{proof}
	
	\subsection{Proof of \thmref{thm-main-equiv-PSL(n,H)}}
	Recall that every skew-involution in $\mathrm{SL}(n,\H) $ is an involution in $\mathrm{PSL}(n,\H) $.  
	The proof of theorem now  follows from \lemref{lem-rev-PSL-eq-conditions}, \thmref{thm-rev-SL-prod-inv-PSL},  and \propref{prop-rev-type-2-PSL-skew-inv-prod}.
	\qed

	\subsection*{Acknowledgment} 
	Gongopadhyay is partially supported by the SERB core research grant CRG/2022/003680. Lohan acknowledges the financial support from the IIT Kanpur Postdoctoral Fellowship. Maity is partially supported by the Seed Grant IISERBPR/RD/OO/2024/23.

\end{document}